\documentclass[10pt,a4paper]{article}
\usepackage[utf8]{inputenc}
\usepackage[T1]{fontenc}
\usepackage{amsmath}
\usepackage{amsthm}
\usepackage{amssymb}
\usepackage{makeidx}
\usepackage{graphicx}
\usepackage[width=16.00cm, height=22.00cm]{geometry}
\usepackage[english]{babel}
\usepackage{caption}
\usepackage[dvipsnames]{xcolor}
\numberwithin{equation}{section}
\usepackage{wrapfig}
\newtheorem{thm}{Theorem}[section]
\newtheorem{lemma}[thm]{Lemma}

\newtheorem{rem}[thm]{Remark}

\newcommand{\ee}{{\mathrm e}}

\newcommand{\ii}{{\mathrm i}}

\definecolor{light-gray}{gray}{0.98}

\DeclareMathOperator{\I}{\mathbb{I}}

\DeclareMathOperator{\OO}{\mathcal{O}}

\title{Third order, uniform in low to high oscillatory coefficients, exponential integrators for Klein-Gordon equations.}

\begin{document}
	
\author{Karolina Kropielnicka\footnote{Institute of Mathematics, Polish Academy of Sciences, Warsaw, Poland}, Karolina Lademann\footnote{Institute of Mathematics, Physics and Computer Science, University of Gda\'{n}sk, Gda\'{n}sk, Poland}}
		\maketitle
\begin{abstract} Allowing for space- and time-dependence of mass in Klein--Gordon equations resolves the problem of negative probability density and of violation of Lorenz covariance of interaction in quantum mechanics. Moreover it extends their applicability to the domain of quantum cosmology, where the variation in mass may be accompanied by high oscillations. In this paper we propose a third-order exponential integrator, where the main idea lies in embedding the  oscillations triggered by the possibly highly oscillatory component intrinsically into the numerical discretisation. While typically high oscillation requires appropriately small time steps, an application of Filon methods allows implementation with large time steps even in the presence of very high oscillation. This greatly improves the efficiency of the time-stepping algorithm. 
	
	Proof of the convergence and its rate are nontrivial and require  alternative representation of the equation under consideration. We derive careful bounds on the growth of global error in time discretisation and prove that, contrary to standard  intuition, the error of time integration does not grow once the frequency of  oscillations increases. Several numerical simulations are presented to confirm the theoretical investigations and the robustness of the method in all oscillatory regimes.
\end{abstract}


\section{Introduction}

In this paper we consider the Klein--Gordon  equation 
\begin{equation}\label{klein_duh}
	\begin{aligned}
		&\frac{\partial ^2}{\partial t^2}  \psi(x,t) =\Delta \psi(x,t)+ m(x,t)\psi(x,t), \quad t \in[t_0,T], \quad x \in  \mathbb{T}^d\\
		& \psi(x,t_0) = \psi_0(x), \quad   \partial_t \psi (x,t_0) = \varphi_0(x),
	\end{aligned}
\end{equation} 
with periodic boundary conditions.\footnote{Our assumption of periodic boundary conditions is solely for the sake of simplicity of exposition and our ideas can be easily generalised to other boundary conditions.} 
Here $\psi(x,t)$ denotes the unknown wave function that we wish to approximate numerically and $-m(x,t)>	0$ stands for the square of the time- and space- dependent mass, which may be presented or approximated by the truncated, modulated Fourier expansion. Thus we assume in this paper that $m(x,t)<0$ is given with
\begin{equation}\label{mass}
	m(x,t) = \sum_{n=1}^N a_n(x,t) e^{i \omega_n t}
\end{equation}
where frequencies $\omega_n \in \mathbb{R}, \: n \leq N, N \in \mathbb{N}$. We let $\tilde{\omega}=\max_{n=1,\ldots,N}|\omega_n|$.  The coefficients $a_n(x,t)$, $n\leq N$, do not oscillate in the $\omega_n$s (i.e. their time derivatives  are bounded independently of $\omega_n$, similarly to the setting in \cite{CHL}).

Linear and nonlinear Klein--Gordon equations received much attention from a theoretical and numerical point of view, see, e.g.\ \cite{bao2, 12,6,RAVIKANTH2009708,2017AIPC.1847b0021K,IKRAM2021,shark,yus}. However, linear Klein--Gordon equations, are typically investigated under assumption of a constant  $m$, and this is  a source of negative probability density. This problem has been overcome in \cite{mos2} and  \cite{mos1}  by admitting space dependence in the input term, that is $m(x)$ instead of $m$. This remedy, however, leads to the violation of Lorentz covariance of the interaction. Only recently these two problems --  negative probability density and of violation of Lorenz covariance -- have been solved in \cite{zno2} and \cite{zno1}, employing $m(x,t)$ instead of $m(x)$. Thanks to this improvement, as explained in \cite{zno3} and \cite{zno1}, application of (\ref{klein_duh}) can be extended also to the domain of quantum cosmology. For more details we refer the reader to \cite{Kofman1994rk}, where the Klein-Gordon  equation (with time- and space- dependent mass  oscillatory in time) is applicable to the theory of reheating the Universe after inflation.

Numerical investigation of (\ref{klein_duh}) has been very limited and none of the aforementioned results is concerned with the case of time- and space-dependent input term.  As far as we are aware, a computational approach  for Klein--Gordon equations with time- and space-dependent input term $m(x,t)$ was  considered only in three papers. In \cite{baderblanes} the authors proposed 4th and 6th order methods based on commutator-free, Magnus type expansion. These powerful methods have been designed for non-oscillatory versions of $m(x,t)$. % and work very efficiently as long as the the time step $h$ scales at most like the oscillations of $m(x,t)$. % , where too large time steps do not suffice for high oscillations, while very small time steps slow down the computations dramatically. 
The methods proposed in \cite{baderblanes} are likely to fall short in case of highly oscillatory input terms (mainly because of wired-in Gauss--Legendre quadrature for highly oscillatory integrals). Highly oscillatory coefficients of type $m(x,t) = \sum_{n=1}^N a_n(x,t) e^{i \omega_n t}$ with $|\omega_n| \gg 1$ were treated in \cite{asympt}, where authors proposed numerical-asymptotic expansions. The latter work, however fails for non-oscillatory components. The most challenging form of coefficient $m(x,t)$ is when it includes low and high frequencies, for example $m(x,t)=a_0(x,t)+a_1(x,t)\ee^{i t}+a_2(x,t)\ee^{i 10^6 t}$, because this prevents us from treating the equations solely with either classical or asymptotic numerical means. 

An efficient approach to this type of problems has been introduced in \cite{accurate_KKK} where the authors utilise splitting techniques based on the Magnus expansion. Efficiency of the proposed method is described in terms of accuracy, rather than  order, and substantially depends on the ratio of magnitude between the time step $h$ and the maximum frequency $\tilde{\omega}$ in $m(x,t)$.

In the current paper we describe a third-order in time method which is equally efficient in the full range of frequencies, e.g.\ $m(x,t)=a_0(x,t)$, $m(x,t)=a_2(x,t)\ee^{i 10^6 t}$ and $m(x,t)=a_0(x,t)+a_1(x,t)\ee^{i t}+a_2(x,t)\ee^{i 10^6 t}$, where the error constant does not grow with the size of $\omega_n$s, and no ratio between time step $h$ and either $\tilde{\omega}$ or space discretisation need be imposed. In a nutshell, the main achievement of the paper is to design a framework for discretizations which are not degraded in the presence of high oscillation. This is achieved mainly due by embedding the oscillatory phases explicitly into the integral followed by an application of  Filon-type methods. We will present a rigorous proof of  convergence, discuss the structure of the error of the method and explain why it is not affected even by extremely high oscillations and conclude by numerical simulations illustrating the results. 

The term {\it low or high oscillations} should be understood as the relation between the time step $h$ and minimum and maximum the frequencies ($\omega_{\min}$ and $\omega_{\max}$ respectively). In the case of $ h \leq \frac{1}{\omega_{\max}} $ we face the low-oscillatory regime. In this case of large $ \omega_{\max} $ requires very small time step $h$, what leads to computational errors and high time and memory costs. Highly oscillatory regime means that  $ h \geq \frac{1}{\omega_{\min}} > \frac{1}{\omega_{\max}} $. In this case methods based on Taylor's theorem allow for the $ n $th-order convergence rate, but lead to the large constant error which scales like $ (\omega_{\max})^n $.

Note that this paper is focused on time integration and the elaborated methodology is suitable for any choice of spatial discretisation. This  may depend on the nature of boundary conditions of the specific problem in applications and is consistent e.g.\ with finite difference, finite element and spectral approaches. In our numerical examples we assumed that the problem, consistently with (\ref{klein_duh}), is equipped with periodic initial and boundary conditions and used a Fourier spectral collocation in space.
%Indeed, as will be shown the error constant does not grow with the size of $\omega_n$. Moreover time step does not need to be related with frequencies $\omega_n$ or space step.  {\cyan ...well... the nearly same can be explained for \cite{accurate_KKK} ! The truth is that \cite{accurate_KKK} is more efficient, but requires a certain ratio between the sparsity of the grids in time and in space} \\

{ \bf Outline of the paper.}
In  Section \ref{DERIVATION} we recall basic properties of the Duhamel formula and  indicate how they can be adapted to Klein--Gordon-type equations. Section \ref{CONVERGENCE} presents the comprehensive analysis of the convergence of the method. In Section \ref{EXPERIMENTS}  we illustrate our theoretical findings  with  numerical experiments. In addition, we compare  various  existing methods with the new approach and highlight the favorable behavior of the new method with respect to high frequencies  $\omega_n$. 

\section{Derivation of the method}\label{DERIVATION}

As already stated in the introduction, we are concerned solely with time discretisation: in practical application it should be of course accompanied by a space  discretisation, but our framework is consistent with an arbitrary choice of the latter. For this reason in the further part of this manuscript we  suppress dependance on space variable $x$ and  write $\psi(t), a_n(t), m(t)$ instead of $\psi(x,t), a_n(x,t), m(x,t)$. In order to derive third order methods we assume that $\psi'''\in L_1([t_0,T])$. % exists a.e. on $[t_0,T]$.}
In this setting we can express (\ref{klein_duh})  as the following abstract evolution equation %(i.e. suppressing dependance on $x$)
\begin{equation}\label{dd}
\dfrac{d}{dt} \left[ \begin{array}{c}
	\psi(t) \\
	\psi'(t)
\end{array} \right]  = \left[ \begin{array}{cc}
	0 & \I \\
	\Delta & 0
\end{array} \right] \left[ \begin{array}{c}
	\psi(t)\\
	\psi'(t)
\end{array} \right] + \left[ \begin{array}{c}
	0\\
	m(t)\psi(t)
\end{array} \right] 
\end{equation}
with formal solution  given by Duhamel formula,
\begin{equation}\label{duh0}
\left[ \begin{array}{c}
	\psi(t) \\
	\psi'(t)
\end{array} \right]  = R(t-t_0) \left[ \begin{array}{c}
	\psi_0 \\
	\varphi_0
\end{array} \right] + \int_{t_0}^{t}R(t-\tau) \left[ \begin{array}{c}
	0 \\
	m(\tau)\psi(\tau)
\end{array} \right] d\tau, 
\end{equation}
where
$$ R(t) = \exp \left[ \begin{array}{cc}
0 & \I t \\
\Delta t& 0
\end{array} \right] = \left[ \begin{array}{cc}
\cos (t G )&  G^{-1} \sin(t  G ) \\
- G	\sin(t G ) & \cos (t G) 
\end{array} \right], \quad \text{ and } \quad G \cdot G = - \Delta.$$
Note that, $-\Delta$ being positive definite, we can also choose a positive  definite $G$.

In computational implementation we resort to full discretisation and then functions $\psi(t), a_n(t), m(t)$ are replaced by vectors and the unbounded operator $G$   by a suitable positive\footnote{We should  use symmetric spatial discretisation of $-\Delta$ to ensure the existence of positive discrete operator $G_d$.  }, finite-dimensional operator $G_d$.

Let us assume that $T/h=K$ and $t_k=kh$. Iterating Duhamel's formula we obtain the following relations at the time steps $t_{k + 1}= t_k + h$, $k=0,\ldots, K-1$:
\begin{eqnarray}\label{duh3}
\nonumber
\left[ \begin{array}{c}
	\psi(t_k +h) \\
	\psi'(t_k +h)
\end{array} \right]  &=& R(t_k +h - t_k) \left[ \begin{array}{c}
	\psi(t_k) \\
	\psi'(t_k)
\end{array} \right] + \int_{t_k}^{t_k +h}R(t_k +h-\tau) \left[ \begin{array}{c}
	0 \\
	m(\tau)\psi(\tau)
\end{array} \right] d\tau \\
&=& R(h) \left[ \begin{array}{c}
	\psi(t_k) \\
	\psi'(t_k)
\end{array} \right] + \int_{0}^{h}R(h-\tau) \left[ \begin{array}{c}
	0 \\
	m(t_k+\tau)\psi(t_k+\tau)
\end{array} \right] d\tau.
\end{eqnarray}
As in the case of Gautschi-type methods \cite{hochlub}, the  above representation of the solution is the foundation of our numerical scheme.

\begin{rem}\label{primy}
In this manuscript we understand that $\psi'(t_k)$ and $\psi''(t_k)$ are the first and second, respectively, time derivatives of function $\psi(\,\cdot\,,t)$ at point $t_k$. In particular we write $\psi'(t_0)$ instead of $\varphi_0$.
\end{rem}

The system (\ref{duh3})  reduces to two  time-stepping formulas for the values of the functions $ \psi (t) $ and  $ \psi' (t) $ at the time step $ t_{k + 1} $, namely
\begin{align}\label{duh6}
\psi(t_k + h) &= \cos (h G ) \psi(t_k) + G^{-1} \sin(h  G ) \psi'(t_k) + \int_{0}^{h} G^{-1} \sin((h-\tau)  G )  m(t_k+\tau)\psi(t_k+\tau)  d\tau,
\\ \label{duh81}
\psi'(t_k + h) &= - G	\sin(h G )  \psi(t_k) + \cos (h G)  \psi'(t_k) + \int_{0}^{h} \cos ((h-\tau) G)  m(t_k+\tau)\psi(t_k+\tau) d\tau.
\end{align} 

The main difficulty of the above formulation is that it is an implicit system of equations. Given the initial conditions we  resort within the integral to the Taylor expansion  %where for certain $ \xi \in [t_k,t_k+\tau] $ we have 
\begin{equation}\label{tay_3}
\psi(t_k+\tau)  = \psi(t_k) + \tau \psi'(t_k) + \dfrac{\tau^2}{2!} \psi''(t_k)  + E_k(\tau)%(t_k,t_k+\tau), % \dfrac{\tau^3}{3!} \psi'''(\xi).
\end{equation} 
where $ E_k(\tau)=\int_{t_k}^{t_k+\tau} \dfrac{(t_k+\tau-s)^2}{2!} \psi'''(s)ds$. %there is such $\xi_k \in[t_k,t_k+\tau]$, that $E_k(\tau)=\dfrac{\tau^3}{3!} \psi'''(\xi_k)$, or $ E_k(\tau)=\int_{t_k}^{t_k+\tau} \dfrac{(t_k+\tau-s)^2}{2!} \psi'''(s)ds$. {\magi (to decide)}

\begin{rem}\label{second_derivative_stability}
The above approximation could be directly incorporated into  (\ref{duh6})  and (\ref{duh81}) because $ \psi''(t_k) $  can be immediately recovered from the initial condition (for $k=0$) and from the original problem (\ref{klein_duh}) for $k\geq 1$. The latter, however would lead to additional numerical evaluations. For this reason for $k\geq 1$ we  resort again to the Taylor formula %, where for certain $  \zeta \in [t_{k-1}, t_k] $ we have
\begin{equation}\label{key}
	\psi''(t_k) =  \dfrac{\psi'(t_k)-\psi'(t_{k-1})}{h} + \bar{E}_k. %\frac{h}{2}	\psi'''(\zeta).
\end{equation} 
where $\bar{E}_k=\frac{1}{h}\int_{t_{k-1}}^{t_{k}} (s-t_{k-1}) \psi'''(s)ds$. %there is $  \zeta_k \in [t_{k-1}, t_k] $ such that $\bar{E}_k=\frac{h}{2}	\psi'''(\zeta)$, or $\bar{E}_k=\frac{1}{h}\int_{t_{k-1}}^{t_{k}} (s-t_{k-1}) \psi'''(s)ds$ {\magi (to decide)}
\end{rem}

\begin{rem}
We further observe that the errors $E_k(\tau)$ and $\bar{E}_k$ contain the third derivative of the unknown function $\psi$, which obviously depends on oscillations $\omega_n$. Indeed, due to the structure of the equation (\ref{klein_duh}) expression $\partial_t^3\psi$ involves $\partial_tm$, so we formally have that
$$
\partial_t^3\psi=\mathcal{O}(\nabla^3\psi)+\mathcal{O}(\tilde{\omega}\psi).
$$
Given that $\partial_t^3\psi$ grows to infinity together with $\tilde{\omega}$, it is natural to expect that $E_k(\tau)$ and $\bar{E}_k$ would also grow to infinity and that these values would affect the overall error constant of approximation. However, as will be shown in Lemma \ref{omega_does_not_matter}   the error of the method does not depend on the oscillations.
\end{rem}

Application of (\ref{tay_3}) and (\ref{key}) in (\ref{duh6}),  (\ref{duh81}) leads directly to a numerical scheme $ \Xi^{[3]} $ presented in Table \ref{TABLE_1}, where $\psi_k$ and $\psi'_k$ are numerical approximations of $\psi(t_k)$ and $\psi’(t_k)$, respectively.

\begin{center} \colorbox{light-gray}{
	\hspace*{-20pt}\begin{tabular}{l}
		Numerical  scheme	$ \Xi^{[3]} $\\ 
		\hline \\
		$   \psi_0=\psi(t_0),\quad\psi'_0=\psi'(t_0),\quad\psi''_{0} =\Delta \psi(t_0) - m(t_0)\psi(t_0) $,\\ \smallskip 
		$\psi_1  = \cos (h G) \psi(t_0)  + G^{-1} \sin(h  G ) \psi'(t_0)  + \int_{0}^{h} G^{-1} \sin((h-\tau)  G )m(t_0+\tau)\left[\psi_0 + \tau \psi'_0 + \frac{\tau^2}{2}\psi''_{0}  \right] d\tau $ \\ \smallskip
		$ \psi'_1  = - G\sin(h G )\psi(t_0)  + \cos (h G) \psi'(t_0)   + \int_{0}^{h} \cos ((h-\tau) G) m(t_0+\tau)\left[\psi_0 + \tau \psi'_0 + \frac{\tau^2}{2}\psi''_{0} \right] d\tau $\\ \smallskip
		\textbf{for }$  k=1,N-1 $ \textbf{do }\\ \smallskip
		\hspace*{10pt}$ \psi_{k+1} = \cos (h G) \psi_k + G^{-1} \sin(h  G ) \psi'_k + \int_{0}^{h} G^{-1}\sin((h-\tau)  G )m(t_k+\tau)\left[\psi_k+ \tau \psi'_k+ \frac{\tau^2}{2h}\left(\psi'_k - \psi'_{k-1} \right)\! \right]\! d\tau $ \\ \smallskip
		\hspace*{10pt}$ \psi'_{k+1} = - G	\sin(h G )\psi_k + \cos (h G) \psi'_k + \int_{0}^{h} \cos ((h-\tau) G) m(t_k+\tau)\left[\psi_k+ \tau \psi'_k+ \frac{\tau^2}{2h}\left(\psi'_k - \psi'_{k-1} \right) \!\right] \!d\tau $ \\
		\textbf{end do} \\
		\hline 
\end{tabular}}
\captionof{table}{Algorithm $\Xi^{[3]} $ (time integration) for finding the approximate solution $\psi_k$ and $\psi'_k$ in the time interval $[t_0,T]$ at points $t_k=t_0+hk$, $h=(T-t_0)/N$. Important part of this pseudocode are the integrals. In the case when they cannot be computed analytically, we recommend Filon-type methods described in the further part of this paper. Practical approximation of the integrals is presented in Appendix \ref{app}. Discretisation in space is general and depends on boundary conditions. In the numerical examples, presented in Section \ref{EXPERIMENTS}, we assume periodic initial and boundary conditions and use Fourier collocation for space discretisation.
}
\label{TABLE_1}
\end{center}

\begin{rem}\label{uqadratures}
If the integrals appearing in numerical scheme $\Xi^{[3]}$ cannot be computed analytically, they need to be approximated. It is important to emphasize that the high performance of numerical approach $ \Xi^{[3]}$ and its third-order convergence are obtained only once we approximate the possibly highly-oscillatory integrals (containing $m(t_k+\tau)$) with a quadrature of order at least $4$ (locally) where the error constant does not depend on the oscillatory parameters $\omega_n$ and where no relation between $h$ and $\tilde{\omega}$ is required. For this reason standard  quadratures like Gauss--Legendre are unequal to the task (because their error constants depend  on the time derivatives of the integrants, in particular on the highly oscillatory part) and we need a more specialised approach. Here we propose to use the  Filon-type methods which obtain desired order in time step $h$, whose error constants do not depend on (possibly) large  $\tilde{\omega}$ and where no relation between $h$ and $\tilde{\omega}$ need be imposed.
\end{rem}

Filon methods are based on the one-dimensional oscillatory integral,
\begin{align}\nonumber
I_{\omega}[f] = 	\int_{a}^{b} f(s) g_{\omega}(s) ds
\end{align}
where $ g_{\omega}(s)  $ is an $ \omega $-oscillatory function, by which we broadly mean a function that oscillates rapidly in $ \omega $. Its organising principle is to  approximate  the non-oscillatory function $ f(s) $ by a polynomial $p(s)$ such that 
$$ \forall i \in \{ 0,1,...,r\} \quad  p^{(i)}(a) = f^{(i)}(a) , \quad p^{(i)}(b) = f^{(i)}(b). $$ 

For Filon-type methods we refer the reader to \cite{article}  and \cite{Iserles2005EfficientQO}, where they have been introduced and comprehensively analysed. 
In our case it is enough to use the quadratic polynomial,
\begin{equation}\label{filon}
p(s) = \frac{f'(h)-f'(0)}{2h}s^2+f'(0)s+f(0) .
\end{equation}

\begin{rem} \label{filon_1}
In the scheme $ \Xi^{[3]}$ we  apply a  (locally) fourth-order Filon-type method to the integrals of the form $\int_0^hf(s)\ee^{i\omega s}ds$. 
It is easy to verify that for $p(s)$ defined as in (\ref{filon}) we have
\begin{align}\nonumber
	|f(s) - p(s)| & \leq C h^3 \max_{\xi \in [0,h]} | f'''(\xi)|,
\end{align}
where $C$ is independent of $\omega\gg1$. This, in turn, means that
\begin{align}\label{Filon_explanation}
	\int_{0}^{h} f(s) \ee^{i\omega_n s} ds  &= \int_{0}^{h} \left[ \frac{f'(h)-f'(0)}{2h}s^2+f'(0)s+f(0)\right]  \ee^{i\omega s} ds + \OO(h^4) ,
\end{align}
where the integral on the right hand side can be evaluated explicitly by integration by parts.
\end{rem}

Note that, unless a Filon method (or similar highly oscillatory quadrature) is used, the method will fail unless the time step $h>0$ is minute, essentially $h<\tilde{\omega}^{-1}$. This phenomena is illustrated in Example 1, see Figure \ref{fig:0}.

\section{The convergence of the method}\label{CONVERGENCE}

It is not straightforward to conclude convergence and stability from the equations in scheme  $\Xi^{[3]}$, because naive calculations of global error lead to exponential growth in number of time steps taken in the interval $[t_0,T]$. For this reason we need to reformulate the equations (\ref{duh6}) and (\ref{duh81}).

\begin{thm}\label{equivalence}
Given that equation (\ref{klein_duh}) is investigated on time interval $[t_0,T]$, $K\cdot h=T$, $t_k=k h$,$ k=0,1,...,K $ and  $G \cdot G = - \Delta$, formulas  (\ref{duh6}) and (\ref{duh81}) are equivalent to 
\begin{align}\label{duh6together}
	\psi(t_k)  = &\cos (k h G ) \psi(t_0) + G^{-1} \sin(k h  G )  \psi'(t_0) + \\ \nonumber
	&\sum_{l=0}^{k-1}G^{-1} \int_{0}^{h} \sin(((k-l)h-\tau)  G )m(t_l+\tau)\psi(t_l+\tau) d\tau,\ k=1,\ldots,K.   \\ \label{duh81together}	
	\psi'(t_{k}) = & - G	\sin(khG) \psi(t_0)  + \cos(kh  G ) \psi'(t_0)  + \\ \nonumber
	&\sum_{l=0}^{k-1}\int_{0}^{h} \cos (((k-l)h-\tau) G) m(t_l+\tau)]\psi(t_l+\tau) d\tau,\ k=1,\ldots,K.
\end{align}
\end{thm}

\begin{proof}
We proceed by induction. One can easily check that the formulas for $\psi(t_{1})$ and $\psi'(t_1)$  obtained by (\ref{duh6together}) and (\ref{duh81together}) are equivalent to these obtained by (\ref{duh6}) and (\ref{duh81}), respectively. Let us assume that the same holds for $\psi(t_{k})$ and $\psi'(t_k)$ for certain $1\leq k\leq K-1$. Then
\begin{align*}
	\psi(t_{k+1})  &= \cos (h G ) \psi(t_k)  + G^{-1} \sin(h  G ) \psi'(t_k) + \int_{0}^{h} G^{-1} \sin((h-\tau)  G )m(t_k+\tau) \psi(t_k+\tau) d\tau \\
	& = \cos (h G ) \left[ \cos (k h G ) \psi(t_0) + G^{-1} \sin(k h  G )  \psi'(t_0)\right]  \\ \nonumber
	%& +  \cos (h G )G^{-1} \int_{0}^{h} \sin((kh-\tau)  G )m(t_0+\tau)\psi(t_k+\tau)  d\tau   \\ \nonumber
	& +  \cos (h G ) \sum_{l=0}^{k-1}G^{-1} \int_{0}^{h} \sin(((k-l)h-\tau)  G )m(t_l+\tau)\psi(t_l+\tau) d\tau \\
	&+G^{-1} \sin(h  G ) \left[- G	\sin(khG) \psi(t_0)  + \cos(kh  G ) \psi'(t_0) \right] \\ \nonumber
	%& +  \int_{0}^{h} \cos ((kh-\tau) G) m(t_0+\tau) d\tau \\ \nonumber
	&  + G^{-1} \sin(h  G )\sum_{l=0}^{k-1}\int_{0}^{h} \cos (((k-l)h-\tau) G) m(t_l+\tau)\psi(t_l+\tau) d\tau \\
	& + \int_{0}^{h} G^{-1} \sin((h-\tau)  G )m(t_k+\tau)\psi(t_k+\tau) d\tau. 
\end{align*}
$G$ being positive definite, we can apply the  trigonometric identities, 
\begin{align*}
	\cos (h G )\cos (k h G ) - [G^{-1} \sin(h  G )][ G	\sin(khG) ]&= \cos((k+1)hG )\\
	\cos (h G )G^{-1} \sin(k h  G )+ G^{-1} \sin(h  G ) \cos(kh  G ) & = G^{-1} \sin((k+1)h  G )
\end{align*} 
and conclude with 
\begin{align*}
	\psi(t_{k+1})  &= \cos((k+1)hG ) \psi(t_0) + G^{-1} \sin((k+1) h  G )  \psi'(t_0) \\ \nonumber
	%	& + G^{-1} \int_{0}^{h} \sin(((k+1)h-\tau)  G )m(t_0+\tau)\left[\psi(t_0) + \tau \psi'(t_0) + \frac{\tau^2}{2}\psi''(t_0)+ \frac{\tau^3}{6}\psi'''({\xi})\right]  d\tau   \\ \nonumber
	& + \sum_{l=0}^{k}G^{-1} \int_{0}^{h} \sin(((k+1-l)h-\tau)  G )m(t_l+\tau) \psi(t_l+\tau) d\tau.
	%	& + \int_{0}^{h} G^{-1} \sin((h-\tau)  G )m(t_k+\tau)\left[\psi_{t_k} + \tau \psi'_{t_k} + \frac{\tau^2}{2h}(\psi'_{t_k}-\psi'_{t_k}) +  \frac{\tau^3}{6}\psi'''(\xi) \right] d\tau 
\end{align*}
This proves equivalence of the values of function $\psi$ obtained by (\ref{duh6}) and (\ref{duh6together}).  A similar result for $ \psi'$ obtained by (\ref{duh81}) and (\ref{duh81together}) can be shown in a similar manner.
\end{proof}

\begin{rem}\label{equivalence_short_method}
The equations from scheme $\Xi^{[3]}$ presented in Table \ref{TABLE_1} can be presented in the similar form to (\ref{duh6together}) and (\ref{duh81together}), namely
\begin{align}\label{eq10s6together}
	\psi_k  = &\cos (k h G ) \psi_0 + G^{-1} \sin(k h  G )  \psi'_0  \\ \nonumber
	+&G^{-1} \int_{0}^{h} \sin((kh-\tau)  G )m(t_0+\tau)\left[\psi_0 + \tau \psi'_0 +\frac{\tau^2}{2}\psi''_{0} \right] d\tau   \\ \nonumber
	+&\sum_{l=1}^{k-1}G^{-1} \int_{0}^{h} \sin(((k-l)h-\tau)  G )m(t_l+\tau)\left[\psi_l + \tau \psi'_l+ \frac{\tau^2}{2h}\left(\psi'_l - \psi'_{l-1} \right) \right] d\tau,   \\ \label{eq20stogether}	
	\psi'_k = & - G	\sin(khG) \psi_0  + \cos(kh  G ) \psi'_0   \\ \nonumber
	+&\int_{0}^{h} \cos ((kh-\tau) G) m(t_0+\tau)]\left[\psi_0 + \tau \psi'_0 +\frac{\tau^2}{2}\psi''_{0} \right] d\tau
	\\ \nonumber
	+&\sum_{l=1}^{k-1}\int_{0}^{h} \cos (((k-l)h-\tau) G) m(t_l+\tau)]\left[\psi_l + \tau \psi'_l+ \frac{\tau^2}{2h}\left(\psi'_l - \psi'_{l-1} \right) \right] d\tau,
\end{align}
where $k=2,\ldots,K$.
\end{rem}

\begin{thm}\label{local_error}
Let $E_k(\tau)=\int_{t_k}^{t_k+\tau} \frac{(t_k+\tau-s)^2}{2!} \psi'''(s)ds$, $k=0,\ldots,K-1$, $\bar{E}_0=0$, and $\bar{E}_k=\frac{1}{h}\int_{t_{k-1}}^{t_{k}} (s-t_{k-1}) \psi'''(s)ds$, $k=1,\ldots,K-1$. The global errors $\varepsilon_k:=\psi(t_k)-\psi_k$ and $\varepsilon_k':=\psi'(t_k)-\psi_k'$, $k=1,\ldots,K$, of method $\Xi^{[3]}$ can be expressed with formulas

\begin{align}\nonumber %\label{eq10s6together}
	\varepsilon_1=\,&\mathcal{R}_1,\\ \nonumber
	\varepsilon_{k}=\,&\varepsilon_{k-1} \\ \nonumber
	&\mbox{}+2G^{-1}\sin\left(\frac{h}{2}  G \right)\sum_{l=1}^{k-2} \int_{0}^{h} \cos\left((k-l)h-\frac{h}{2}-\tau)  G \right)m(t_l+\tau)\left[\varepsilon_{l}+\tau\left(1+\frac{\tau}{2h}\right)\varepsilon'_{l}-\frac{\tau^2}{2h}\varepsilon'_{l-1}\right]d\tau   \\ \nonumber
	\,&\mbox{}+\int_{0}^{h} G^{-1}\sin\left((h-\tau) G \right)m(t_{k-1}+\tau)\left[\varepsilon_{k-1}+\tau\left(1+\frac{\tau}{2h}\right)\varepsilon'_{k-1}-\frac{\tau^2}{2h}\varepsilon'_{k-2}\right]d\tau+ \mathcal{R}_k,\quad k\geq2,\\[5pt] \nonumber %\sum_{l=0}^{k-2}\bar{R}_{k,l}+R_{k,k},   \\ \nonumber
	\varepsilon'_1=\,&\mathcal{R}’_1,\\ \nonumber
	\varepsilon_{k}' =\,& \varepsilon_{k-1}'   \\ \nonumber
	\,&\mbox{}-2\sin\left(\frac{h}{2}  G \right)\sum_{l=1}^{k-2}\int_{0}^{h} \sin\left((k-l)h-\frac{h}{2}-\tau)  G \right) m(t_l+\tau)]\left[\varepsilon_{l}+\tau\left(1+\frac{\tau}{2h}\right)\varepsilon'_{l}-\frac{\tau^2}{2h}\varepsilon'_{l-1}\right]d\tau \\ \nonumber
	\,&\mbox{}+\int_{0}^{h} \cos\left((h-\tau) G \right) m(t_{k}+\tau)\left[\varepsilon_{k-1}+\tau\left(1+\frac{\tau}{2h}\right)\varepsilon'_{k-1}-\frac{\tau^2}{2h}\varepsilon'_{k-2}\right]d\tau+\mathcal{R}'_k,\quad k\geq2,
\end{align}
where
\begin{align}\nonumber 
	\mathcal{R}_{k}=\, & 2G^{-1}\sin\left(\frac{h}{2}  G \right)\sum_{l=1}^{k-2} \int_{0}^{h} \cos\left((k-l)h-\frac{h}{2}-\tau)  G \right)m(t_{l}+\tau)\left[\frac{\tau^2}{2}\bar{E}_{l}+E_{l}(\tau)\right] d\tau\\ \label{remainder}
	\,&\mbox{}+\int_{0}^{h} G^{-1}\sin((h-\tau)  G )m(t_{k-1}+\tau)\left[\frac{\tau^2}{2}\bar{E}_{k-1}+E_{k-1}(\tau)\right] d\tau,
\end{align}
\begin{align}\nonumber
	\mathcal{R}_{k}' =  & -2\sin\left(\frac{h}{2}  G \right)\sum_{l=1}^{k-2} \int_{0}^{h} \sin\left((k-l)h-\frac{h}{2}-\tau)  G \right)m(t_{l}+\tau)\left[\frac{\tau^2}{2}\bar{E}_{l}+E_{l}(\tau)\right] d\tau\\ \label{remainder_prim}
	\,&\mbox{}+ \int_{0}^{h} \cos((h-\tau)  G )m(t_{k-1}+\tau)\left[\frac{\tau^2}{2}\bar{E}_{k-1}+E_{k-1}(\tau)\right] d\tau
\end{align}
\end{thm}

\begin{proof}
Based on the Theorem \ref{equivalence} and Remark \ref{equivalence_short_method} the global errors $\varepsilon_k$ and $\varepsilon'_k$ can be expressed as

\begin{align} \label{ve_formula}
	\varepsilon_k  = \,&G^{-1} \int_{0}^{h} \sin((kh-\tau)  G )m(t_0+\tau)\left[\psi(t_0+\tau)-\psi_0 - \tau \psi'_0 - \frac{\tau^2}{2}\psi''_0  \right]d\tau   \\ \nonumber
	\,&\mbox{}+\sum_{l=1}^{k-1}G^{-1} \int_{0}^{h} \sin(((k-l)h-\tau)  G )m(t_l+\tau)\left[\psi(t_l+\tau)-\psi_l-\tau\psi'_l- \frac{\tau^2}{2h}\left(\psi'_l - \psi'_{l-1} \right)\right]d\tau,   \\  \label{ve_formula_prim}	
	\varepsilon_k' = \,&\int_{0}^{h} \cos ((kh-\tau) G) m(t_0+\tau)]\left[\psi(t_0+\tau)-\psi_0 - \tau \psi'_0 - \frac{\tau^2}{2}\psi''_{0}  \right] d\tau
	\\ \nonumber
	\,&\mbox{}+\sum_{l=1}^{k-1}\int_{0}^{h} \cos (((k-l)h-\tau) G) m(t_l+\tau)]\left[\psi(t_l+\tau)-\psi_l-\tau\psi'_l- \frac{\tau^2}{2h}\left(\psi'_l - \psi'_{l-1} \right)\right]  d\tau.
\end{align}
Recalling the initial condition,  $\psi(t_0)=\psi_0$, $\psi'(t_0)=\psi'_0$, and  the fact that $\psi''_0$ is restored from the equation (cf.\ the first line of scheme $\Xi^{[3]}$), we  observe that
\begin{align}\nonumber
	\psi(t_0+\tau)-\psi_0-\tau\psi'_0-\frac{\tau^2}{2}\psi''_0\,&= E_0(\tau),\\ \nonumber
	\psi(t_{l}+\tau)-\psi_{l}-\tau\psi'_{l}- \frac{\tau^2}{2h}\left(\psi'_{l} - \psi'_{l-1} \right)\,&= \varepsilon_{l}+\tau\varepsilon'_{l}+\frac{\tau^2}{2h}\varepsilon'_{l}-\frac{\tau^2}{2h}\varepsilon'_{l-1}+\frac{\tau^2}{2}\bar{E}_{l}+E_{l}(\tau),\quad l=1,\ldots,k-1.\nonumber
\end{align}
Next, recalling that $\bar{E}_{0}=0$ and letting 
\begin{align}\label{remainders}
	R_{k,l}= \,& \int_{0}^{h} G^{-1}\sin(((k-l)h-\tau)  G )m(t_{l}+\tau)\left[\frac{\tau^2}{2}\bar{E}_{l}+E_{l}(\tau)\right] d\tau,\\ \nonumber
	R_{k,l}' = \,& \int_{0}^{h} \cos(((k-l)h-\tau)  G )m(t_{l}+\tau)\left[\frac{\tau^2}{2}\bar{E}_{l}+E_{l}(\tau)\right] d\tau
\end{align}
for $k=1,\ldots,K$ and $l=0,\ldots,k-1$, we easily conclude that
\begin{align}\nonumber %\label{eq10s6together}
	\varepsilon_k  = %&G^{-1} \int_{0}^{h} \sin((kh-\tau)  G )m(t_0+\tau)\left[\psi(t_0+\tau)-\psi(t_0) - \tau \psi'(t_0) - \frac{\tau^2}{2}\psi''_{t_0}  \right]d\tau,   \\ \nonumber
	\,&\sum_{l=1}^{k-1}G^{-1} \int_{0}^{h} \sin(((k-l)h-\tau)  G )m(t_l+\tau)\left[\varepsilon_{l}+\tau\left(1+\frac{\tau}{2h}\right)\varepsilon'_{l}-\frac{\tau^2}{2h}\varepsilon'_{l-1}\right]d\tau  + \sum_{l=0}^{k-1}R_{k,l},   \\ \nonumber
	\varepsilon_k' = %&\int_{0}^{h} \cos ((kh-\tau) G) m(t_0+\tau)]\left[\psi(t_0+\tau)-\psi(t_0) - \tau \psi'(t_0) - \frac{\tau^2}{2}\psi''_{t_0}  \right] d\tau, \\ \nonumber
	\,&\sum_{l=1}^{k-1}\int_{0}^{h} \cos (((k-l)h-\tau) G) m(t_l+\tau)]\left[\varepsilon_{l}+\tau\left(1+\frac{\tau}{2h}\right)\varepsilon'_{l}-\frac{\tau^2}{2h}\varepsilon'_{l-1}\right]  d\tau + \sum_{l=0}^{k-1}R'_{k,l}.
\end{align}
Given that
\begin{align}\nonumber  
	\sin(((k-l)h-\tau)  G )-\sin(((k-1-l)h-\tau)  G )\,&=2\sin\left(\frac{h}{2}  G \right)\cos\left((k-l)h-\frac{h}{2}-\tau)  G \right)\!, \\ \nonumber
	\cos(((k-l)h-\tau)  G )-\cos(((k-1-l)h-\tau)  G )\,&=-2\sin\left(\frac{h}{2}  G \right)\sin\left((k-l)h-\frac{h}{2}-\tau)  G \right)
\end{align}
and subtracting consecutive global errors (\ref{ve_formula}) and (\ref{ve_formula_prim}) their growth can be represented in the form
\begin{align}\nonumber %\label{eq10s6together}
	\!&\varepsilon_{k}-\varepsilon_{k-1}   \\ \nonumber
	\,&=\sum_{l=1}^{k-2} \int_{0}^{h} 2G^{-1}\sin\left(\frac{h}{2}  G \right)\cos\left((k-l)h-\frac{h}{2}-\tau)  G \right)m(t_l+\tau)\left[\varepsilon_{l}+\tau\left(1+\frac{\tau}{2h}\right)\varepsilon'_{l}-\frac{\tau^2}{2h}\varepsilon'_{l-1}\right]d\tau  \\ \nonumber
	\,&\mbox{}+\int_{0}^{h} G^{-1}\sin\left((h-\tau) G \right)m(t_{k-1}+\tau)\left[\varepsilon_{k-1}+\tau\left(1+\frac{\tau}{2h}\right)\varepsilon'_{k-1}-\frac{\tau^2}{2h}\varepsilon'_{k-2}\right]d\tau\\\nonumber
	\,&\mbox{}+ \sum_{l=0}^{k-2}\left[R_{k,l}-R_{k-1,l}\right]+ R_{k,k-1}\\ \nonumber %\sum_{l=0}^{k-2}\bar{R}_{k,l}+R_{k,k},   \\ \nonumber
	\,&\varepsilon_{k}' -\varepsilon_{k-1}' \\ \nonumber
	\,&=-\sum_{l=1}^{k-2}\int_{0}^{h} 2\sin\left(\frac{h}{2}  G \right)\sin\left((k-l)h-\frac{h}{2}-\tau)  G \right) m(t_l+\tau)]\left[\varepsilon_{l}+\tau\left(1+\frac{\tau}{2h}\right)\varepsilon'_{l}-\frac{\tau^2}{2h}\varepsilon'_{l-1}\right]d\tau \\ \nonumber
	\,&\mbox{}+\int_{0}^{h} \cos\left((h-\tau) G \right) m(t_{k}+\tau)\left[\varepsilon_{k-1}+\tau\left(1+\frac{\tau}{2h}\right)\varepsilon'_{k-1}-\frac{\tau^2}{2h}\varepsilon'_{k-2}\right]d\tau \\
	\,& + \sum_{l=0}^{k-2}\left[R'_{k,l}-R'_{k-1,l}\right]+ R'_{k,k-1}
\end{align}
where, telescoping
\begin{align}\nonumber 
	\mathcal{R}_{k}:=\sum_{l=0}^{k-2}\left[R_{k,l}-R_{k-1,l}\right]+ R_{k,k-1} \quad {\rm and} \quad \mathcal{R}'_{k}:=\sum_{l=0}^{k-2}\left[R'_{k,l}-R'_{k-1,l}\right]+ R'_{k,k-1},
\end{align}
we derive the claim of the lemma.
\end{proof}

\begin{lemma}\label{omega_does_not_matter}
Let $\mathcal{R}=\max_{k=1,\ldots,K}\|\mathcal{R}_k\|$ and $\mathcal{R}'=\max_{k=1,\ldots,K}\|\mathcal{R}_k’\|$, where $\mathcal{R}_k$ and $\mathcal{R}’_k$ are defined by (\ref{remainder}) and (\ref{remainder_prim}) in Theorem \ref{local_error}. Then
\begin{align}\nonumber
	\mathcal{R}&\leq Ch^3 \min\{h \tilde{\omega}, 1\}, \\ \nonumber
	\mathcal{R}'&\leq C\|G\|h^3 \min\{h \tilde{\omega}, 1\},
\end{align}
where $C$ depends on $N$, $\max_{n\leq N,t\in[t_0,T]}|a_n(t)|$,  $\max_{n\leq N,t\in[t_0,T]} |a_n’(t)|$,  $ \max_{t\in[t_0,T]}|\psi(t)|$,  $ \max_{t\in[t_0,T]}|\psi’(t)|$,  $ \max_{t\in[t_0,T]}||\nabla\psi’(t)||$ and  $T$. 
\end{lemma}

\begin{proof}
Functions $\mathcal{R}_k$ and $\mathcal{R}_k'$, $k=1,\ldots,K$, depend on the errors $E_k(\tau)$ and $\bar{E}_k$, which in turn depend on the third derivative of the unknown function,
$$
\psi'''(s)=\Delta\psi'(s)+\ii \sum_{n=1}^N \omega_n a_n(s)e^{\ii\omega_n s}\psi(s)+\sum_{n=1}^N a_n'(s)e^{\ii\omega_n s}\psi(s)+\sum_{n=1}^N a_n(s)e^{\ii\omega_n s}\psi'(s).
$$
The second term of the above expression involves multiplications by $\omega_n s$, which require careful attention. For this reason we write  $E_k(\tau)=E^0_k(\tau)+E^\omega_k(\tau)$, $k=0,\ldots,K$, where 
\begin{align}\nonumber
	E^0_k(\tau) & = \int_{t_k}^{t_k+\tau} \dfrac{(t_k+\tau-s)^2}{2!} \left[\Delta\psi'(s)+\sum_{n=1}^N a_n'(s)e^{\ii\omega_n s}\psi(s)+\sum_{n=1}^N a_n(s)e^{\ii\omega_n s}\psi'(s)\right]ds,\\ \nonumber
	E^\omega_k(\tau) & = \int_{t_k}^{t_k+\tau} \dfrac{(t_k+\tau-s)^2}{2!}  \sum_{n=1}^N \ii\omega_n a_n(s)e^{\ii\omega_n s}\psi(s)ds,
\end{align}
and  $\bar{E}_0=0$, $\bar{E}_k=\bar{E}^0_k+\bar{E}_k^\omega$, $k=1,\ldots,K$, where
\begin{align}\nonumber
	\bar{E}_k^0(\tau) & = \frac{1}{h}\int_{t_{k-1}}^{t_{k}} (s-t_{k-1}) \left[\Delta\psi'(s)+\sum_{n=1}^N a_n'(s)e^{\ii\omega_n s}\psi(s)+\sum_{n=1}^N a_n(s)e^{\ii\omega_n s}\psi'(s)\right]ds,\\ \nonumber
	\bar{E}_k^\omega & = \frac{1}{h}\int_{t_{k-1}}^{t_{k}} (s-t_{k-1}) \sum_{n=1}^N \ii \omega_n a_n(s)e^{\ii\omega_n s}\psi(s)ds.
\end{align}
It is trivial to observe, that $E_k^0(\tau)\sim\tau^3C_k^0(\tau)$, $k\geq0$, and  $\bar{E}_k^0(\tau)\sim h\bar{C}_k^0$, ($k\geq1$), where $C_k^0(\tau)$ and $\bar{C}_k^0$ do not depend on the $\omega_n$s. One can also easily notice that $E_k^\omega(\tau)\sim\tau^3\sum_{n=1}^N \omega_n C_k^\omega(\tau)$ and that $\bar{E}_k^0(\tau)\sim h\sum_{n=1}^N \omega_n\bar{C}_k^\omega$, where $C_k^\omega(\tau)$ and $\bar{C}_k^\omega$ also do not depend on frequencies $\omega_n$ either. Integration by parts of $E_k^\omega(\tau)$ yields 
\begin{align} \nonumber
	&\sum_n \dfrac{ \ii \omega_n}{2!} \int_{t_k}^{t_k+\tau} (t_k+\tau-s)^2a_n(s)e^{\ii\omega_n s}\psi(s)ds
	\\
	\nonumber
	=& \left| \begin{array}{ll}
		(t_k+\tau-s)^2a_n(s)\psi(s)	&  e^{\ii\omega_n s}\\
		-2(t_k+\tau-s)a_n(s)\psi(s)+(t_k+\tau-s)^2(a_n(s)\psi(s))'_s		& \frac{1}{i \omega_n} e^{i \omega_n s}
	\end{array}\right|	 \\
	\nonumber
	=& \dfrac{1}{2!}\sum_n \left[-\tau^2a_n(t_k)\psi(t_k)e^{\ii\omega_n{t_k}} +\int_{t_k}^{t_k+\tau}2(t_k+\tau-s)a_n(s)\psi(s)e^{i \omega_n s}ds\right] \\
	\nonumber
	&\mbox{}-\frac{1}{2!}\sum_n \int_{t_k}^{t_k+\tau}(t_k+\tau-s)^2(a_n(s)\psi(s))'_se^{i \omega_n s}ds.
\end{align}
We represent $\bar{E}_k^\omega$ in a similar manner. This leads to the conclusion that $E_k^\omega(\tau)\sim\tau^2 C_k^1(\tau)$ and that $\bar{E}_k^0(\tau)\sim \bar{C}_k^1$, where $C_k^1(\tau)$ and $\bar{C}_k^1$  do not depend on frequencies $\omega_n$. 

Concluding,
$$
E_k(\tau)\leq C\min\{\tau^3  \tilde{\omega}, \tau^2\}  \quad {\rm and} \quad \bar{E}_k\leq C\min\{h\tilde{\omega}, h^0\}, \quad k\geq1,
$$
where the constant $C$ is independent of $\omega_n$s but depends on benign quantities like $N$,  $\max_{n\leq N,t\in[t_0,T]}|a_n(t)|$,  $\max_{n\leq N,t\in[t_0,T]} |a_n’(t)|$,  $ \max_{t\in[t_0,T]}|\psi(t)|$,  $ \max_{t\in[t_0,T]}|\psi’(t)|$,  $ \max_{t\in[t_0,T]}||\nabla\psi’(t)||$.

%Concluding
%$$
%E_k(\tau)\sim\min\{\tau^3 \sum_{n=1}^N \omega_n, \tau^2\}C_k,\ k\geq0,  \quad {\rm and} \quad \bar{E}_k\sim\min\{h\sum_{n=1}^N \omega_n, h^0\}\bar{C}_k, \ k\geq1,
%$$
%where $\|\bar{C}_k(\tau)\|$ and $\|\bar{C}_k\|$ do not depend on $\sum_{n=1}^N\omega_n$.\\

Next we focus our attention on the structure of $\mathcal{R}_k$ and $\mathcal{R}'_k$. Obviously 

$$
\frac{\tau^2}{2}\bar{E}_{k}+E_{k}(\tau)\sim\min\{h^3 \sum_{n=1}^N \omega_n, h^2\}\tilde{C}_k,\ k\geq0, 
$$
Moreover we  observe that 
$\|2G^{-1}\sin\left(\frac{h}{2}  G \right)\|\leq h$ and $\|2\sin\left(\frac{h}{2}  G \right)\|\leq h\|G\|$. Equipped with this estimates we easily conclude the claim of the Lemma.
\end{proof}

\begin{rem}
Obviously the unboundedness of operator $G$ may raise concerns, because of the presence of  $\|G\|$ in the estimate. We recall, though, that the theme of this paper is  time integration and that eventually space discretization must be combined with our narrative. In practice, this means that in applications $G$ will be approximated with some bounded operator $G_d$, in particular that $\|G_d\|$ will be bounded.
\end{rem}

Finally we are ready to formulate and prove the theorem on the convergence of our method.

\begin{thm}\label{convergence_short_method}
The numerical method $\Xi^{[3]}$ is convergent, its global error $\|u(T)-u_T\|$ does not grow with $\tilde{\omega}$ and it exhibits third (global) order of accuracy in time, meaning that $\|u(T)-u_T\| \leq C h^3$ as $h\rightarrow 0$. 
\end{thm}

\begin{proof}

Let $M:=\max_{t\in[t_0,T]}|m(t)|$. Based on the Theorem \ref{local_error} and Lemma \ref{omega_does_not_matter} we obtain for  $k\geq2$ the following inequalities
\begin{align}\nonumber
	\|\varepsilon_k\|  \leq  & \|\varepsilon_{k-1}\|+M\Big(h\|\varepsilon_{k-1}\|+h^2\|\varepsilon'_{k-1}\|+h^2\|\varepsilon'_{k-2}\| \Big)+\mathcal{R},\\ \nonumber
	\|\varepsilon'_k\|  \leq  & \|\varepsilon'_{k-1}\|+M\|G\|\Big(h\|\varepsilon_{k-1}\|+h^2\|\varepsilon'_{k-1}\|+h^2\|\varepsilon'_{k-2}\| \Big)+\mathcal{R}'.
\end{align}
%with $\mathcal{R}_{k-1}=\left\|\int_{0}^{h}\left(\frac{\tau^2}{ 2}\bar{E}_{k-1}+E_{k-1}(\tau)\right)d\tau\right\|$.
Let now $\rho_k$ and $\rho_k'$ be the upper bounds of $\|\varepsilon_k\|$ and $\|\varepsilon_k'\|$ respectively, $k=0,\ldots,K$, therefore
\begin{equation}\label{rho_0}
	\rho_0= 0,\quad 
	\rho_0'= 0,\quad
	\rho_1= \|\mathcal{R}_1\|\leq\mathcal{R},\quad 
	\rho_1'= \|\mathcal{R}’_1\|\leq\mathcal{R}’,
\end{equation}
and for $k\geq2$
\begin{align}\label{rho_1}
	\rho_k  =  & \rho_{k-1}+M\Big(h\rho_{k-1}+h^2\rho'_{k-1}+h^2\rho'_{k-2} \Big)+\mathcal{R},\\ \label{rho_2}
	\rho'_k  =  & \rho'_{k-1}+M\|G\|\Big(h\rho_{k-1}+h^2\rho'_{k-1}+h^2\rho'_{k-2} \Big)+\mathcal{R}'.
\end{align}
We conclude that $\|\varepsilon_{k}\|\leq\rho_{k}$, $\|\varepsilon'_{k}\|\leq\rho'_{k+1}$, $\rho_{k}\leq\rho_{k+1}$, $\rho'_{k}\leq\rho'_{k+1}$ and that  the following inequalities are satisfied:
\begin{align}\nonumber
	\rho_k  \leq  & \rho_{k-1}+M\left(h\rho_{k-1}+2h^2\rho'_{k-1}\right)+\mathcal{R},\\ \nonumber
	\rho_k'  \leq  & \rho_{k-1}'+M\|G\|\left(h\rho_{k-1}+2h^2\rho'_{k-1}\right)+\mathcal{R}'. 
\end{align}\\
Let us recall now that $\|G\|\geq1$. Based on (\ref{rho_0})-(\ref{rho_2}) we observe that $\mathcal{R}\leq\mathcal{R}’$, $\rho_k\leq\rho'_k $ which in turn means that
\begin{align}\nonumber
	\rho'_k  \leq  &\left(1+M\|G\|(h+2h^2)\right)\rho'_{k-1}+\mathcal{R}’.   %\\ 
\end{align}

	\begin{align}\nonumber
		\rho'_K  \leq  & \mathcal{R}'\sum_{j=0}^{K-1}\left(1+(h+2h^2)M\|G\|\right)^{j}\leq \mathcal{R}'\sum_{j=0}^{K-1}\left(1+M\|G\|3h\right)^j 	=  \mathcal{R}' \frac{1-\left(1+M\|G\|3h\right)^{K}}{1-\left(1+M\|G\|3h\right)} \\ \nonumber
		= & \mathcal{R}' \frac{-\sum_{j=1}^{K} {K \choose j}\left(M\|G\|3h\right)^{j}}{-M\|G\|3h} 	= \mathcal{R}' \frac{-\sum_{j=2}^{K} {K \choose j}\left(M\|G\|3h\right)^{j}-{K \choose 1}\left(M\|G\|3h\right)}{-M\|G\|3h}  \\ \nonumber
		&=   \mathcal{R}' K +\mathcal{R}' \frac{-\sum_{j=2}^{K} {K \choose j}\left(M\|G\|3h\right)^{j}}{-M\|G\|3h} = 	\mathcal{R}'\left(K  +3TM\|G\|\mathcal{O}(h^0)\right),
	\end{align}

and  
\begin{align}\label{estimate}
	\|\varepsilon_K\|\leq \|\varepsilon’_K\|\leq \rho'_K  \leq  & \ TC\|G\|\min\left\{h^3\tilde{\omega},h^2\right\},
\end{align}
where $C$ depends on $N$,$M$, $\max_{n\leq N,s\in[t_0,T]}|a_n(s)|$, $\max_{n\leq N,s\in[t_0,T]}|a_n'(s)|$, $\max_{s\in[t_0,T]}|\psi(s)|$, $\max_{s\in[t_0,T]}|\psi'(s)|$, $\max_{s\in[t_0,T]}|\psi''(s)|$ and on $\max_{s\in[t_0,T]}\|\nabla\psi'(s)\|$.

This means that global error of the method $\|\psi(T)-\psi_T\|$ is bounded uniformly in $\tilde{\omega}$ by at least $h^2$ and that in the limit as $h\rightarrow0$ we obtain third order global error.
\end{proof}

%%%%%%%%%%%%%%%%%%%%%%%%%%%%%%%%%%%%%%%

\section{Numerical examples}\label{EXPERIMENTS}

We thereby focus on two equations. In Example 1  we  consider the wave equation with a highly oscillating term of single frequency  $\omega$  and discuss behaviour of numerical methods with regard to the size of this parameter. We are concerned with the effectiveness of scheme $\Xi^{[3]}$ depending on the type of quadratures used for approximation of the integrals appearing in the scheme. Moreover we  compare various numerical approaches and show that each of them is effective in a different oscillatory regime. In Example~2 we  consider the wave equation with multiple frequencies and run comparisons of different time stepping methods. It will be seen that the proposed third-order method performs exceedingly well in this setting. 

In all the  experiments we assume that the solution is confined to the spatial interval $ [x_0 , x_M ]$ and assume periodic boundary conditions. We divide the interval into $ M = 200 $ sub-intervals of length $ \Delta  x = (x_M - x_0)/M$ and discretise spatial operators with Fourier collocation method, described in detail in \cite{kopriva}.
As  a reference solution we apply to the semi-discretised equations a 6th order method based on self-adjoint basis of Munthe-Kaas  \& Owren \cite{blanesros}  with a  step size $ h = 10^{-6} $. We carry out  numerical integration for presented method using various time steps and measure the $ l_2 $ absolute error at the final time $T$. The errors are plotted in  double-logarithmic scale.\\

\noindent
\textit{Example 1.} \\

\noindent Let us consider the following wave equation with time- and space-dependent potential 
\begin{eqnarray}\label{example_3}
\partial_t^2 u &=&  \partial_x^2 u - \left[  1+\cos (\omega t) \right] x^2 u, \quad x\in [-10,10] , \quad t \in [0, 1], \\
\nonumber
u(x,0) &=& e^{- \frac{x^2}{2}} , \quad \partial_t u (x,0) = 0, \\
\nonumber
u(-10,t) &=& u(10,t).
\end{eqnarray}

In Figure \ref{fig:0}   parameter $\omega$ takes values $10^3$ (on the left), $10^4$ (in the centre) and $10^5$ (on the right), respectively. The solution to the equation (\ref{example_3}) is approximated with the $\Xi^{[3]}$ method, with the integrals  computed with the fourth-order Filon method (as presented in Remark \ref{filon_1}). The derived solution is compared against the results of the $\Xi^{[3]}$ method complemented by Gauss--Legendre quadratures of fourth, sixth and eighth order, respectively. As can be seen, even high orders of Gauss--Legendre quadratures fall well short of the third (globally) order method unless the time step $h$ is significantly smaller than the frequency $\omega$. This confirms our claim that Filon methods (or other highly oscillatory quadrature methods) need  be applied in the presence of  high oscillation.

\begin{figure}[tbh]
\centering
\includegraphics[width = 1\linewidth]{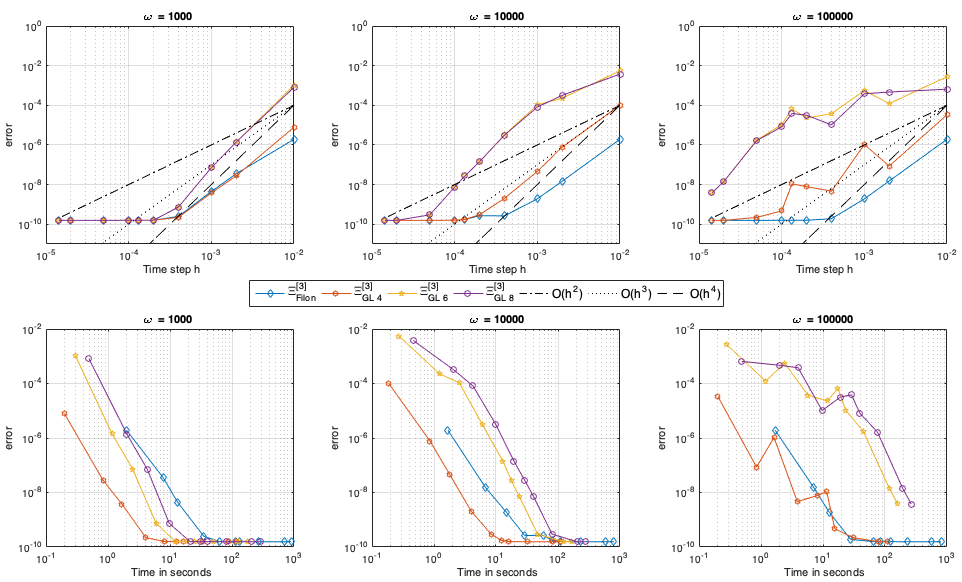}
\caption{Accuracy and time of computation in seconds of the numerical solution to problem (\ref{example_3}), where results have been obtained with $\Xi^{[3]}$ method with applications of Filon-type method and high order Gauss--Legendre quadratures.}
\label{fig:0}
\end{figure} 

\newpage

In Figure \ref{fig:4} we  compare the method $\Xi^{[3]}$ against other schemes from the literature, like well known 2nd and 4th order Runge--Kutta methods (${\rm RK}^{[2]} $ and ${\rm RK}^{[4]} $). We will also consider the 4th-order method $ \Sigma_{3c}^{[4] }$ from \cite{baderblanes} (denoted here as ${\rm BBCK}^{[4]}$) which is designed for non-oscillatory potentials, and  the 3rd-order asymptotic method from \cite{asympt} (denoted by ${\rm Asympt}^{[3]}$), appropriate for extremely high oscillations.\footnote{The method from \cite{accurate_KKK} is yet unpublished. While based on entirely different premises, it is  competitive with the approach of the present paper.} The comparisons are analysed in three regimes, with $\omega=10$, $\omega=500$ and $\omega=1500$. As can be seen method ${\rm Asympt}^{[3]}$ proves it efficiency only in presence of high oscillations and fails in the case of $\omega=10$. All other methods work very well in slowly oscillatory regime. Second order Runge--Kutta method loses its effectiveness already for $\omega=500$. It is clear from the middle and right columns that the second and  fourth order methods ${\rm RK}^{[2]} $,${\rm RK}^{[4]} $ and ${\rm BBCK}^{[4]}$ require time step $h$ to be smaller than $\omega^{-1}$. The new method $\Xi^{[3]}$ maintains third global error and proves its effectiveness in all (from slowly to highly) oscillatory regimes and does not require the time step $h$ to be smaller than $\omega^{-1}$. We can also observe, that the error constant does not grow with the growth of $ \omega $.

\begin{figure}[thb]
\centering
\includegraphics[width=1\linewidth]{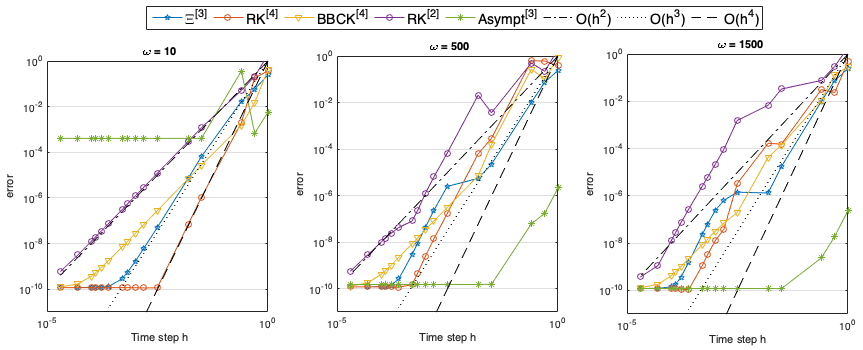}
\caption{Comparison of accuracy of methods known in the literature, where three oscillatory regimes in equation \ref{example_3}  are considered. }
\label{fig:4}
\end{figure}

\noindent\textit{Example 2.}\\

\noindent In this example we consider an equation where the input term features wide variety of frequencies. 
\begin{eqnarray}
\label{example_4}
\partial_t^2 u &=&  \partial_x^2 u - \sum_{n=0}^{5} \left(  1+\cos (10^n t) \right) x^2 u,\quad x\in [-10,10] , \quad t \in [0, 1], \\
\nonumber
u(x,0) &=& e^{- x^2/2} , \quad \partial_t u (x,0) = 0, \\
\nonumber
u(-10,t) &=& u(10,t).
\end{eqnarray} 
\newpage
\begin{wrapfigure}{r}{7cm}
\includegraphics[width=8cm]{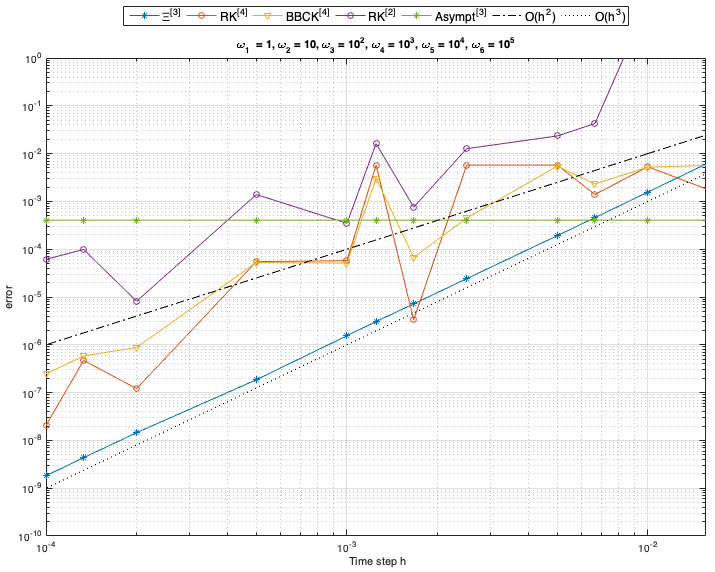}
\caption{The accuracy of numerical methods applied to the multi-frequency problem (\ref{example_4}) where the input term features a range of frequencies from $\omega_1=1$ to $\omega_6=10^5$.}
\label{fig:omegas}
\end{wrapfigure}

According to our computations the error of the new method $ \Xi^{[3]} $ scales like $ h^3 $ globally uniformly in $ \omega $, see Figure \ref{fig:omegas}. Note that the error constant is not affected by $\tilde{\omega}=10^5$, a very high frequency, and that the new method behaves well for all time steps. Moreover, as expected, the asymptotic method performs very purely. This is caused by the small frequency $ \omega_1  = 1 $ which appears in the input term. On the other hand the large oscillations like $ \omega_6  = 10^5$ sabotage methods  that approximate poorly highly oscillatory integrals. 

Our two examples demonstrate vividly the value of our proposed method $\Xi^{[3]}$ in comparison with both classical and asymptotic methods.
\\ \ \\ \ \\ \ \\ \ \\ \ \\ \ \\  \ \\ 

\section{Acknowledgments}

We are grateful to the reviewers for suggestions that significantly improved the presentation of the article.
We also wish to thank Thomas Zlosnik from University of Gdańsk  for enlightening discussions in the field of quantum cosmology and  importance of Klein-Gordon equations with oscillatory time- and space- dependent mass.
The work in this project was financed by the NCN project no. 2019/34/E/ST1/00390 and  partially  by Simons Foundation Award No. 663281 granted to the Institute of Mathematics of the Polish Academy of Sciences for the years 2021–2023. Numerical simulations were carried out by Karolina Lademann at the Academic Computer Centre in Gda\'{n}sk.

\appendix
\def\theequation{\Alph{section}.\arabic{equation}}
\section{Filon method in numerical scheme $\Xi^{[3]}$ }\label{app}

In the numerical scheme $\Xi^{[3]}$ we have to approximate the following two integrals
\begin{align} \nonumber
& \int_{0}^{h} G^{-1}\sin((h-\tau)  G )m(t_k+\tau)\left[\psi_k+ \tau \psi'_k+ \frac{\tau^2}{2h}\left(\psi'_k - \psi'_{k-1} \right)\! \right]\! d\tau,
\\ \nonumber
& \int_{0}^{h} \cos ((h-\tau) G)  m(t_k+\tau)\left[\psi_k+ \tau \psi'_k+ \frac{\tau^2}{2h}\left(\psi'_k - \psi'_{k-1} \right)\! \right]\! d\tau,
\end{align} 
where $m(t_k+\tau) = \sum_{n=1}^N a_n(t_k+\tau) e^{i \omega_n (t_k+\tau)}$.

Let us explain the approximation of Filon-type on the example of the first one. Thus we start with introducing the  following notation
\begin{align}\nonumber
f_{1,n}(\tau) &= G^{-1} \sin((h-\tau)  G )\; a_n(t_k+\tau)\\
\nonumber
f_{2,n}(\tau) &= G^{-1} \sin((h-\tau)  G )\; a_n(t_k+\tau)\; \tau, \\
\nonumber
f_{3,n}(\tau) &= G^{-1} \sin((h-\tau)  G )\; a_n(t_k+\tau) \; \frac{\tau^2}{2h},
\end{align}
and observe that

\begin{align*}
\int_{0}^{h} &G^{-1} \sin((h-\tau)  G )  m(t_k+\tau)  \left[\psi_k+ \tau \psi'_k+ \frac{\tau^2}{2h}\left(\psi'_k - \psi'_{k-1} \right)\! \right]\!  d\tau \\
\nonumber
& =\psi_k\int_{0}^{h} \sum_{n=1}^N f_{1,n}(\tau) e^{i \omega_n (t_k+\tau)}   d\tau +\psi'_k \int_{0}^{h}\sum_{n=1}^N f_{2,n}(\tau) e^{i \omega_n (t_k+\tau)}    d\tau +\left(\psi'_k - \psi'_{k-1} \right) \int_{0}^{h}\sum_{n=1}^N f_{3,n}(\tau) e^{i \omega_n (t_k+\tau)}   d\tau.
\end{align*}

Defining 
\begin{align}\nonumber
\texttt{int}_{1,n}(t_k)  & :=\int_0^h e^{i \omega_n (t_k+\tau)} d\tau  = -\frac{i(-1+e^{i \omega_n h })e^{i \omega_n t_k } }{\omega_n},\\
\nonumber
\texttt{int}_{2,n}(t_k) & := \int_0^h e^{i \omega_n (t_k+\tau)} \tau d\tau  = \frac{(-1+e^{i \omega_n h }(1-ih \omega_n))e^{i \omega_n t_k } }{\omega_n^2}, \\
\nonumber
\texttt{int}_{3,n}(t_k) &:=\int_0^h e^{i \omega_n (t_k+\tau)} \tau^2 d\tau  = \frac{(e^{ih\omega_n} (-ih^2 \omega_n^2 + 2h \omega_n +2i)-2i)e^{i \omega_n t_k } }{\omega_n^3}, 
\end{align}

we obtain the following approximation
\begin{align*}
\int_{0}^{h} & G^{-1} \sin((h-\tau)  G )  m(t_k+\tau) \left[\psi_k+ \tau \psi'_k+ \frac{\tau^2}{2h}\left(\psi'_k - \psi'_{k-1} \right)\! \right]\!  d\tau \\
\nonumber	
& =\sum_{n=1}^N  \left(f_{1,n}(0)\psi_k+f_{2,n}(0)\psi'_k+f_{3,n}(0)(\psi'_k-\psi’_{k-1})\right)\texttt{int}_{1,n}(t_k)\\ \nonumber	
& +\sum_{n=1}^N  \left(f_{1,n}'(0)\psi_k+f_{2,n}’(0)\psi'_k+f_{3,n}’(0)(\psi'_k-\psi’_{k-1})\right)\texttt{int}_{2,n}(t_k) \\
\nonumber	
& +\sum_{n=1}^N  \left(\frac{f_{1,n}'(h)-f_{1,n}'(0)}{2h}\psi_k+\frac{f_{2,n}’(h)-f_{2,n}’(0)}{2h}\psi'_k+\frac{f_{3,n}’(h)-f_{3,n}’(0)}{2h}(\psi'_k-\psi’_{k-1})\right)\texttt{int}_{3,n}(t_k) +\mathcal{O}(h^4).
\end{align*}

\bibliographystyle{plain}

\bibliography{document}

\end{document}